\documentclass[12pt]{amsart}
\usepackage{amsmath, amssymb, multirow, fullpage, graphicx, amsthm}
\newtheorem{theorem}[equation]{Theorem}
\newtheorem{lemma}[equation]{Lemma}

\newtheorem{remark}[equation]{Remark}
\newtheorem{conjecture}[equation]{Conjecture}
\numberwithin{equation}{section}

\def\dz{\hbox{$\mathbb Z$}}

\def\dc{\hbox{$\mathbb C$}}
\def\dq{\hbox{$\mathbb Q$}}
\def\C{\hbox{$\mathcal C$}}
\def\l({\left (}
\def\r){\right )}
\def\Q{\hbox{$\mathbb Q$}}

\def\F{\hbox{$\mathbb F$}}

\newcommand{\mylabel}[1]{\label{#1}}
\newcommand{\myref}[1]{\ref{#1}}

\def\index{n}
\def\heckeIndex{m}
\def\wtf{wt(f)}  
\def\wtg{{wt(g)}}

\begin{document}

\title{Rankin-Cohen brackets of eigenforms and modular forms}

 \author[Beyerl]{Jeffrey Beyerl}
 \address[Beyerl]{
Department of Mathematics\\
University of Central Arkansas\\
Conway, AR 72035
}

\begin{abstract}
We use Maeda's Conjecture to prove that the Rankin-Cohen bracket of an eigenform and any modular form is only an eigenform when forced to be because of the dimensions of the underlying spaces. We further determine when the Rankin-Cohen bracket of an eigenform and modular form is not forced to produce an eigenform and when it is determined by the injectivity of the operator itself. This can also be interpreted as using the Rankin-Cohen bracket operator of eigenforms to create evidence for Maeda's Conjecture.
\end{abstract}

\maketitle

\section{Introduction and Background}

Throughout this paper we will be investigating modular forms of level 1. Every such modular form has a Fourier expansion $f(z)=\sum a_m q^m$ where $q=e^{2\pi i z}$. The space of modular forms of weight $k$ is denoted by $M_k$; the space of cuspidal forms is denoted by $S_k$. Each of these have a diagonal basis consisting of modular forms with leading terms whose $q$-exponents are distinct. The $n^\text{th}$ Rankin-Cohen bracket of two modular forms is
$$[f,g]_\index = \sum_{r+s=\index} (-1)^r \binom{\index+wt(f)-1}{s}\binom{\index+wt(g)-1}{r}f^{(r)}(z)g^{(s)}(z).$$
Here $f^{(a)}(z)$ denotes the normalized derivative: $f^{(a)} = \frac{1}{(2\pi i)^a} \frac{d^a }{dz^a}f$. In particular, the $q$-coefficient is unchanged: $f^{(a)}(z)=a_1 q^1 + \Omega(q^2)$, where $\Omega$ is the standard big-Omega. For more information on the Rankin-Cohen bracket, see \cite{Lanphier08} or \cite{Lanphier04}. Recall that the Rankin-Cohen bracket operator is a bilinear operator from $M_{\wtf}\times M_{\wtg}$ to $M_{\wtf+\wtg+2\index}$. Additionally it is cuspidal if either input is cuspidal, or $\index>0$. If $\index=0$, it is merely pointwise multiplication.

The most common example of an eigenform is probably the Eisenstein series, of which we'll require the use of throughout this paper. The weight $k$ Eisenstein series is $E_k=1-\frac{2k}{B_k} \sum_{n=1}^{\infty} \sigma_{k-1}(n) q^n$ where $B_k$ is a Bernoulli number and $\sigma_{k-1}$ is the sum of the $k^{\text{th}}$ powers of divisors function.

\newpage Properties of products of eigenforms have been studied in various ways, starting in 1999 when Duke and Ghate independently proved exactly when the product of two eigenforms is an eigenform \cite{Duke99, Ghate00}. Since those seminal papers, there have been numerous continuations and generalizations. The most notable works are as follows. The product of many eigenforms was investigated by Emmons and Lanphier \cite{EmmonsLanphier07}. The product of nearly holomorphic modular forms was considered by Beyerl, James, Trentacoste, and Xue \cite{Beyerl12}, with Beyerl, James, and Xue going on to also consider the case that one factor need only be a modular form \cite{Beyerl2014}. The product of eigenforms through the Rankin-Cohen bracket operator was worked out by Lanphier and  Takloo-Bighash \cite{Lanphier04}. Generalizations to higher level have been studied by several authors including Ghate, Emmons, and Johnson \cite{Emmons05, Ghate00, Johnson13}. The Rankin-Cohen Bracket Operator applied to eigenforms was studied by Meher \cite{Meher2012}. In the same paper Meher also proved some results on modular forms with character.

In this paper we continue to extend these results by investigating when the Rankin-Cohen bracket of an eigenform and a modular form is again an eigenform.

\section{Preliminaries}
A subspace $S\subseteq S_k$ is $\F$-rational if it is stable under the action of $Gal(\C/\F$), where each automorphism acts on modular forms through their Fourier coefficients. The following theorem is the tool that relates the existence of equations yielding eigenforms to Maeda's Conjecture. It is proved as Lemma 2.3 in \cite{Beyerl2014}.

\begin{lemma}\mylabel{ProperSubspaceNoEforms}
	If $S$ is a proper $\F$-rational subspace of $S_k$ and $S$ contains an eigenform, then all the Hecke polynomials of weight $k$ are reducible over $\F$. 
\end{lemma}

Throughout the sequel the Hecke operator $T_{m, \wtf+\wtg+2\index}$ will be abbreviated as $T_m$. Similarly the Hecke Polynomial $T_{m, \wtf+\wtg+2\index}(x)$ will be abbreviated as $T_m(x)$. 

Maeda's Conjecture (See Conjecture \myref{Maeda}) tells us that no Hecke operators are reducible. In Theorems \myref{CuspidalCase}, \myref{NonCuspidalCase}, and \myref{fCuspidalCase}, we give sufficient conditions involving the factorization of certain eigenorms that would imply corresponding Hecke operators are reducible. To date no example satisfying these conditions has been found; this gives evidence that Maeda's Conjecture might be true.

We start with a simple but necessary technical detail given in the lemma below.

\begin{lemma}
	Let $\sigma\in Gal(\dc/\dq)$, and $f,g$ be modular forms. Then $\sigma([f,g]_\index)=[\sigma(f), \sigma(g)]_\index$.
\end{lemma}

\begin{proof}
	Recall that $\sigma(f)$ is defined on the Fourier coefficients  of $f$. That is, if $f=\sum f_i q^i$, then $\sigma(f):=\sum \sigma(f_i) q^i$. Also note that $\sigma$ fixes $\dq$, so that in particular it fixes binomial coefficients. We can now compute $\sigma([f,g]_\index)$. 

\begin{align*}
	\sigma([f,g]_n)&=\sigma\l(\sum_{r+s=\index} (-1)^r \binom{n+\wtf-1}{s}\binom{n+\wtg-1}{r} f^{(r)}(z) g^{(s)}(z)\r)\\
	&=\sum_{r+s=n} (-1)^r \binom{n+\wtf-1}{s}\binom{n+\wtg-1}{r} \sigma\l(\sum_{i+j=m} i^r j^s f_i g_j q^m  \r) \\
	&=\sum_{r+s=n} (-1)^r \binom{n+\wtf-1}{s}\binom{n+\wtg-1}{r}\sum_{i+j=m} i^r j^s \sigma(f_i)\sigma( g_j) q^m \\
	&=\sum_{r+s=n} (-1)^r \binom{n+\wtf-1}{s}\binom{n+\wtg-1}{r} (\sigma(f))^{(r)} (z) (\sigma(g))^{(r)}(z) \\
	&=[\sigma(f), \sigma(g)]_n
	\end{align*}

\end{proof}

Together these lemmas allow us to map the entire space of modular or just the cuspidal modular forms to a rational subspace of modular forms via the Rankin-Cohen bracket operator. We flesh this out in the next section. 

\section{Main Results}
In this section we give main results relating the Rankin-Cohen bracket and reducibility of the Hecke operator. Theorems \ref{CuspidalCase}, \ref{NonCuspidalCase}, and \ref{fCuspidalCase} together consider every nontrivial possibility of an eigenform and modular form. Each theorem gives sufficient conditions for $T_m$ to be reducible and/or the the Rankin-Cohen bracket to produce an eigenform.

Each theorem contains the Rankin-Cohen bracket $[f,g]$ or the linear operator $[f,\bullet]_n$. For consistency $f$ is always an eigenform, and $g$ is a modular form. The case where $g$ is also an eigenform is already covered by Lanphier and Takloo-Bighah \cite{Lanphier04} 

We begin by noting that there are four cases depending on whether $f$ and $g$ are cuspidal. If both $f$ and $g$ are cuspidal, clearly $[f,g]_\index$ cannot be an eigenform. This can be seen by looking at the formula for $[f,g]_\index$ in terms of the Fourier coefficients and noticing that both the $q^0$ and $q^1$ coefficients are zero. The remaining three cases each have multiple parts. See the table in Figure \ref{TableOfCases}.

\begin{theorem}\mylabel{CuspidalCase}
Assume $g$ is a cuspidal modular form and $f$ is a noncuspidal eigenform, that is, $f=E_{\wtf}$. There are two cases:
Case 1: $\dim(S_{\wtf+\wtg+2\index})>\dim(S_\wtg)$. If there is a cuspidal modular form $g$ such that $[f,g]_\index$ is an eigenform, then $T_\heckeIndex (x)$ is reducible for all $\heckeIndex$. Case 2: $\dim(S_{\wtf+\wtg+2\index})\leq\dim(S_{\wtg})$. In this case there is a cuspidal modular form $g$ such that $[E_{\wtf}, g]_\index$ is an eigenform. 
\end{theorem}

If we leave $f$ as an Eisenstein series and allow $g$ to be noncuspidal, we obtain the next theorem.

\begin{theorem}\mylabel{NonCuspidalCase}
Assume $g$ is a noncuspidal modular form and $f$ is a noncuspidal eigenform, that is, $f=E_{\wtf}$. There are three cases:
Case 1: $\dim(S_{\wtf+\wtg+2\index})> \dim(M_{\wtg})$. In this case if there is a noncuspidal modular $g$ form such that $[f,g]_\index$ is an eigenform, then $T_\heckeIndex (x)$ is reducible for all $\heckeIndex$.
Case 2: $\dim(S_{\wtf+\wtg+2\index})< \dim(M_{\wtg})$. In this case there is a noncuspidal modular form $g$ such that $[E_{\wtf}, g]_\index$ is an eigenform. 
Case 3: $\dim(S_{\wtf+\wtg+2\index})= \dim(M_{\wtg})$. In this case either case (1) or case (2) will apply, depending on whether or not $[f, \bullet]_\index$ is injective.
\end{theorem}

\begin{figure}
\begin{center}
{\renewcommand{\arraystretch}{1.2}
\begin{tabular}{cc|cc|}
	\cline{3-4}
	\multicolumn{2}{c}{} & \multicolumn{2}{|c|}{$f$} \\ \cline{3-4}
	\multicolumn{2}{c}{} & \multicolumn{1}{|c}{$E_k$} & \multicolumn{1}{|c|}{Cuspidal} \\ \hline
	\multicolumn{1}{|c}{\multirow{2}{*}{$g$}} & \multicolumn{1}{|c}{Noncuspidal}& \multicolumn{1}{|c|}{Theorem \ref{NonCuspidalCase}}& \multicolumn{1}{c|}{Theorem \ref{fCuspidalCase}} \\ \cline{2-4}
	\multicolumn{1}{|c}{} & \multicolumn{1}{|c}{Cuspidal} & \multicolumn{1}{|c|}{Theorem \ref{CuspidalCase}}& \multicolumn{1}{c|}{Trivial} \\ \hline
\end{tabular} 
}
\end{center}\caption{Organization of cases in the main theorems.}\label{TableOfCases}
\end{figure}

The final theorem addresses the scenario that $f$ is noncuspidal while $g$ is cuspidal. 

\begin{theorem}\mylabel{fCuspidalCase}
Let $f$ be a cuspidal eigenform. There are two cases:
Case 1: $\dim(S_{\wtf+\wtg+2\index})> \dim(M_{\wtg})$. In this case if there is a noncuspidal modular $g$ form such that $[f,g]_\index$ is an eigenform then $T_\heckeIndex (x)$ is reducible for all $\heckeIndex$. Case 2: $\dim(S_{\wtf+\wtg+2\index})\leq \dim(M_{\wtg})$. In this case there is a noncuspidal modular form $g$ such that $[E_{\wtf}, g]_\index$ is an eigenform whenever $[f,\bullet]_n$ is injective. 
\end{theorem}

\begin{remark}
When comparing Theorems \ref{NonCuspidalCase} and \ref{fCuspidalCase}, we remark that the question of injectivitity is slightly different in these two cases. In particular in Theorem \ref{NonCuspidalCase} $f$ is noncuspidal, so the Fourier expansion has a constant term. On the other hand in Theorem \ref{fCuspidalCase} $f$ is cuspidal, and so there is no constant term to work with. While the theorems overlap in weight, the operators themselves are different.
\end{remark}

\section{Proofs of Main Thoerems}

We start with the proof of of Theorem \myref{CuspidalCase}.
\begin{proof}[\bf Proof of Theorem \myref{CuspidalCase}]
	If $\index=0$ this is multiplication, and so reduces to our previous result \cite{Beyerl2014}. Otherwise we have that $[E_{\wtf}, S_{\wtg}]_\index\subseteq S_{\wtf+\wtg+2\index}$. Noting that $S_{\wtg}$ has a rational basis (See \cite{Shimura1971}), we have that $\sigma([E_{\wtf}, S_{\wtg}]_\index)=[\sigma(E_{\wtf}), \sigma(S_{\wtg})]_\index = [E_{\wtf}, S_{\wtg}]_n$ and so is rational.

Hence if $\dim(S_{\wtf+\wtg+2\index})>\dim(S_{\wtg})$ we have a proper rational subspace of $S_{\wtf+\wtg+2\index}$. Lemma \myref{ProperSubspaceNoEforms} then tells us that if there are any Eigenforms in this space, all the Hecke Operators $T_\heckeIndex(x)$ are reducible over $\Q$. 

On the other hand suppose $\dim(S_{\wtf+\wtg+2\index})\leq \dim(S_{\wtg})$. Then because $\wtf+2n\geq 4$, $\dim(S_{\wtf+\wtg+2\index})= \dim(S_{\wtg})$.

Then we note that the linear map $[E_{\wtf},\bullet]_\index: S_{\wtg} \to S_{\wtf+\wtg+2n}$ maps a diagonal basis to a diagonal basis; thus it is surjective and therefore must contain an eigenform.
\end{proof}

The proof of the next theorem, \myref{NonCuspidalCase}, handles the situation that both $f$ and $g$ are noncuspidal. It is broken into three cases, depending on the dimensions of the two relevant spaces.

\begin{proof}[\bf Proof of Theorem \myref{NonCuspidalCase}]
	If $n=0$ the Rankin Cohen bracket operator is multiplication and so the theorem reduces to our previous result \cite{Beyerl2014}. Assume $n\geq1$.

In the case that $\dim(S_{\wtf+\wtg+2\index})> \dim(M_{\wtg})$, consider the map $[E_{\wtf},\bullet]_\index: M_{\wtg} \to S_{\wtf+\wtg+2n}$ which gives a proper rational subspace of $S_{\wtf+\wtg+2n}$ so that similar to the proof of Theorem \myref{CuspidalCase} if the range contained an eigenform, then all Hecke polynomials would be reducible. For completeness sake, we will note that the dimension of this space is either $\dim(M_{\wtg})$ or $\dim(S_{\wtg})$, but in both scenarios it is proper and not necessary for the proof.

	In the next case $\dim(S_{\wtf+\wtg+2\index})< \dim(M_{\wtg})$. In fact a diagonal basis of $S_{\wtg}$ stays diagonal when all of $M_{\wtg}$ is mapped, so the dimension can decrease by at most one, and thus $\dim(S_{\wtf+\wtg+2\index})= \dim(S_{\wtg})$. Hence by a similar argument as the previous theorem, there is some $g$ so that $[f,g]_\index$ is an eigenform.
	
	Now consider the case that $\dim(S_{\wtf+\wtg+2\index})= \dim(M_{\wtg})$. When the operator $[E_{\wtf},\bullet]_\index: M_{\wtg} \to S_{\wtf+\wtg+2n}$ is not injective, then the dimension decreases by one and we fall into the same argument as the first case and any eigenform would yield reducibility of all the Hecke operators. When it is injective, then the argument from case (2) applies and there is an eigenform $[f,g]_\index$ for some modular form $g$. There are a total of 149 cases (infinite classes based on the weight of $g$ mod $12$) where this occurs; these cases are illustrated in Figure \ref{Cases149}. 
\end{proof}

\begin{figure}
\includegraphics[scale=.5]{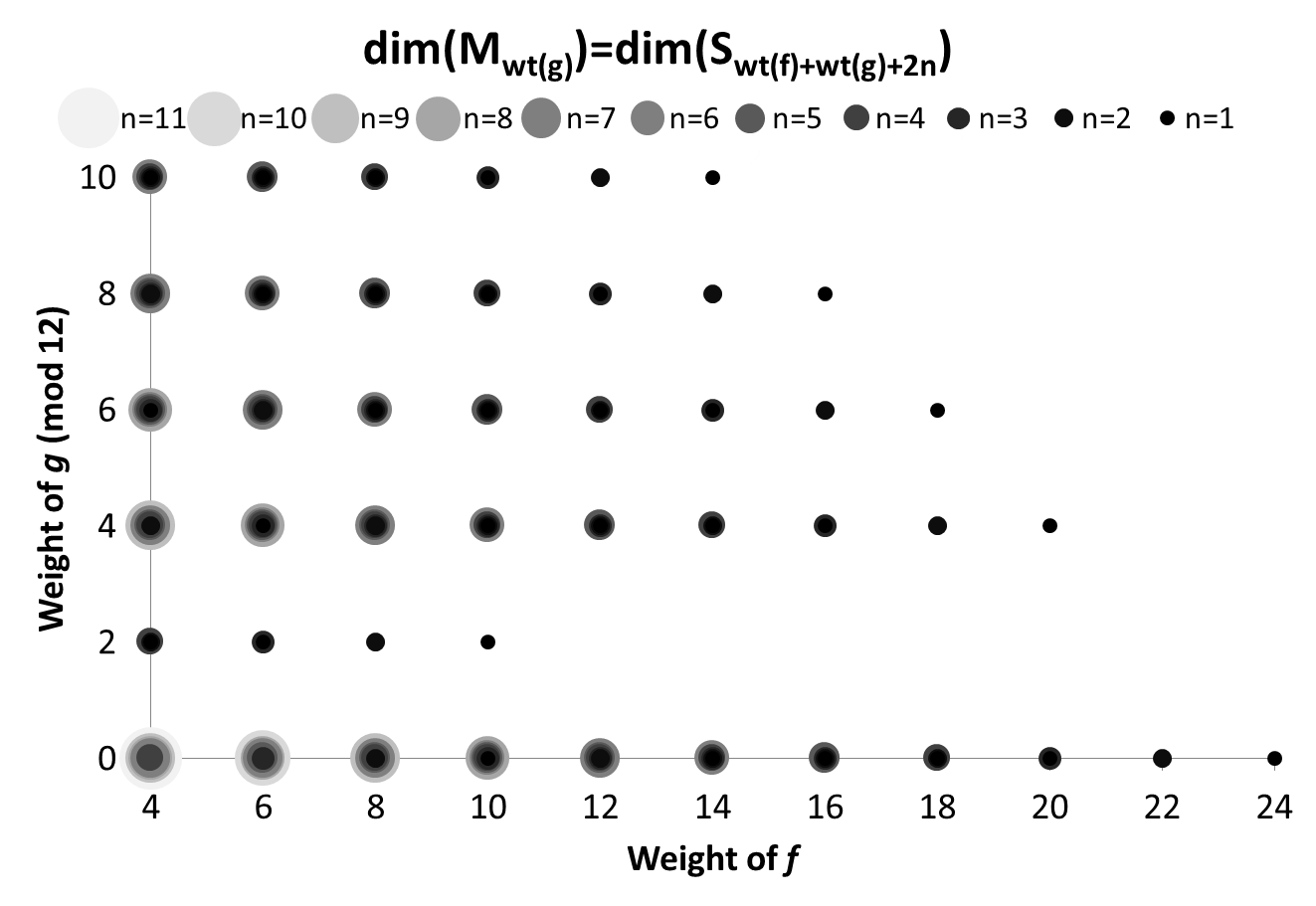}
\caption{Given different weights of $g$ (mod $12$) and $f$, there are only 149 times that $\dim(S_{\wtf+\wtg+2n})= \dim(M_{\wtg})$. This is easily computable, but shown here for illustration. There are 149 circles on this graph. The position of the circle is determined by the weights of $f$ and $g$, while the size and shade are determined by $n$. Note that for illustrative purposes many circles are stacked ontop of each other. For instance at $wt(f)=14$ and $wt(g)=4$ there are two circles. One corresponding to $n=1$ and another to $n=2$. It is straightforward, albeit tedious, to write down a table listing each of these cases}\label{Cases149}
\end{figure}

The final theorem, proven below, assumes that $f$ is cuspidal.

\begin{proof}[\bf Proof of Theorem \myref{fCuspidalCase}]
If $n=0$, this is multiplication and is covered in \cite{Beyerl2014}. Hence we need only consider the cases that $n\geq 1$. 

If the dimension of the underlying spaces increases too much, then $M_{wt(g)}$ creates the a rational subspace $[f,M_{wt(g)}]_\index$ of $S_{\wtf+\wtg+2\index}$. The rational subspace Lemma, \myref{ProperSubspaceNoEforms}, then forces all the Hecke operators $T_\heckeIndex (x)$ to be reducible. This is the case that $\dim(S_{\wtf+\wtg+2\index})> \dim(M_{\wtg})$ and in particular this occurs whenever $\wtf>24$ or $n>7$. 

There are only 45 remaining cases (infinite classes with respect to $wt(g)$ being fixed only modulo 12) where $\dim(S_{\wtf+\wtg+2\index}) = \dim(M_{\wtg})$. These are all in Figure \ref{Cases149} when $f$ could be cuspidal. They can be identified from Figure \ref{Cases149} by ignoring the entries where $\wtf=4, 6, 8, 10,$ and $14$. In each of these cases either the argument from case one of the proof of Theorem \ref{NonCuspidalCase} or the argument from case two of the proof of Theorem \ref{NonCuspidalCase} applies, depending on whether or not $[f, \bullet]_\index:M_{\wtg}\rightarrow S_{\wtf+\wtg+2\index}$ is injective.

The final case (infinite class again with respect to $wt(g)$ varying) is the only case in which $\dim(S_{\wtf+\wtg+2\index}) = \dim(M_{\wtg})-1$. This is when $\wtf=12$ and $\wtg \equiv 0$ mod $12$. The analysis from the previous paragraph is identical, except that the function to consider is whether or not $[f, \bullet]_\index:S_{\wtg}\rightarrow S_{\wtf+\wtg+2\index}$ is injective.
\end{proof}

\section{Conclusions and Maeda's Conjecture}

We have continued the classification of the factorization of eigenforms, in particular in terms of the Rankin-Cohen bracket operator. We now relate this to Maeda's Conjecture, stated in its original form below. 

\begin{conjecture}[Maeda, \cite{Maeda97}]\mylabel{Maeda}
	The Hecke algebra over $\dq$ of $S_k (SL_2(\dz))$ is simple (that is, a single number field) whose Galois closure over $\dq$ has Galois group isomorphic to the symmetric group $\mathcal{S}_n$ (with $n=dim S_k (SL_2(\dz))$).  
\end{conjecture}

This conjecture implies that all Hecke operators are irreducible. In this paper we related the existence of factorizations of eigenforms through the Rankin-Cohen bracket to the dimension of underlying spaces of modular forms. The primary results are phrased as ``Given such a factorization, all Hecke operators of the proper weight are reducible." Maeda's conjecture leads us to believe that no Hecke operators are reducible. Hence we can view these results in the following light: the lack of factorization of eigenforms gives yet more evidence that Maeda's Conjecture is true. Other papers that address evidence of Maeda's Conjecture include \cite{Ghitza2012} and \cite{Wiese2013}. In particular, as a corollary of the latter paper and this work, for weights up to 12000 the Rankin-Cohen bracket operator never produces an eigenform, except when forced to for dimension consideration.

To summarize the main result, we combine the three main theorems into this:

\begin{theorem}
Assuming Maeda's Conjecture and the conditions listed in Theorems \myref{CuspidalCase}, \myref{NonCuspidalCase}, and \myref{fCuspidalCase}, Eigenforms cannot possibly be factored through the Rankin-Cohen bracket in a way such that one of the factors is another eigenform.''
\end{theorem}

\bibliographystyle{plain}
\bibliography{mybib}
\end{document}